\documentclass[a4paper]{amsart}

\usepackage{amsmath}
\usepackage{amssymb}
\usepackage{amsfonts}
\usepackage{mathrsfs}
\usepackage{array}
\usepackage{enumerate}
\usepackage[french, english]{babel}
\usepackage[all]{xy}

\numberwithin{equation}{section}

    \newcommand{\BC}{{\mathbb {C}}} 
     \newcommand{\BF}{{\mathbb {F}}}

    \newcommand{\BQ}{{\mathbb {Q}}} \newcommand{\BR}{{\mathbb {R}}}

     \newcommand{\BZ}{{\mathbb {Z}}}

    \newcommand{\CG}{{\mathcal {G}}} 
     
    \newcommand{\CK}{{\mathcal {K}}} 
     
    \newcommand{\CO}{{\mathcal {O}}}

    \renewcommand{\mod}{\ \mathrm{mod}\ }\renewcommand{\Re}{{\mathrm{Re}}}

    \theoremstyle{plain}
    \newtheorem{thm}{Theorem}[section] \newtheorem{cor}[thm]{Corollary}
    \newtheorem{lem}[thm]{Lemma}  \newtheorem{prop}[thm]{Proposition}

\newtheorem{exm}[thm]{Example}

\theoremstyle{remark} 
\theoremstyle{remark} 
\theoremstyle{remark} 

    \numberwithin{equation}{section}

\DeclareFontFamily{U}{wncy}{}
\DeclareFontShape{U}{wncy}{m}{n}{<->wncyr10}{}
\DeclareSymbolFont{mcy}{U}{wncy}{m}{n}

\begin{document}

\title[Birch-Tate formula and $K_2\CO_F$ for real quadratic number fields $F$]{An application of Birch-Tate formula to tame kernels of real quadratic number fields}
\author{Li-Tong Deng and Yong-Xiong Li}


\begin{abstract}
         Let $F$ be a real quadratic number field with discriminant $D$ and $\CO_F$ the ring of integers in $F$. Let $\chi_F$ be the Dirichlet character associated to $F/\BQ$.
         Write $L(\chi_F,s)$ for the Dirichlet L-function of $\chi_F$.
         By an induction argument for imprimitive Dirichlet L-values, we get several $2$-divisibility results on $L(\chi_F,-1)$ when $D$ has arbitrarily finitely many prime divisors. As an application, by making use of the Birch-Tate formula for $F$, we determine the $2$-primary part for the second $K$ group $K_2\CO_F$. We also give a new proof for an old theorem of Browkin and Schinzel.
\end{abstract}

\maketitle

\selectlanguage{english}

\section{Introduction}

Let $F$ be a real quadratic number field with the ring of integers $\CO_F$.
For a ring $R$, let $K_2 R$ be the Milnor group of $R$. The kernel of the surjective homomorphism
\[K_2 F\to \bigoplus_\mathfrak{p} (\CO_F/\mathfrak{p})^\times,\]
given by the tame symbols at all finite primes $\mathfrak{p}$ of $F$, is called the tame kernel of $F$ and is known to be isomorphic to $K_2\CO_F$.
A theorem of Garland \cite{garland} shows that $K_2\CO_F$ is a finite abelian group. For any prime $\ell$, we write $K_2\CO_F(\ell)$ the $\ell$-primary subgroup of $K_2\CO_F$. In this article, we shall determine an explicit structure for $K_2\CO_F(2)$ for several families of real quadratic number fields $F$ where the discriminant of $F$ can have arbitrarily finitely many prime factors.

\medskip

For any odd prime $\ell$, we denote by $\left(\frac{\cdot}{\ell}\right)$ the Legendre symbol.
The main result in this article is given in the following theorem.

\bigskip

\begin{thm}\label{main-thm-1}
Let $F=\BQ(\sqrt{D})$ be a real quadratic number field with discriminant $D=4\ell_1\cdots\ell_n$ where $\ell_i$ are distinct primes such that $\ell_1\equiv 3\mod 8$ and $\ell_i\equiv 5\mod 8$ for $2\leq i\leq n$. Assume that $\left(\frac{\ell_i}{\ell_j}\right)=-1$ for $2\leq i<j\leq n$, $\left(\frac{\ell_1}{\ell_i}\right)=1$ for $i\geq3$, $\left(\frac{\ell_1}{\ell_2}\right)=-1$ and $n$ is even. Then
\[K_2\CO_F(2)\simeq (\BZ/2\BZ)^{n-1}\times (\BZ/2^\delta\BZ)\]
where $\delta$ is an integer at least $3$.
\end{thm}

\bigskip

We remark that for $n=2$, the above theorem is due to Qin \cite{Q3} (see also \cite{BS} for a lower bound). So the theorem can be viewed as a generalization of Qin's result to discriminant of $F$ with arbitrarily finitely many prime factors. Some numerical data (included in section \S\ref{k2-str}) suggest that $\delta$ is always $3$. However, we currently see no way of proving it theoretically.

\medskip

For a finite abelian group $\mathfrak{A}$ and any integer $n\geq0$, we write $r_{2^n}(\mathfrak{A})$ for the number of cyclic components of $\mathfrak{A}$ of order divisible by $2^n$. The number $r_{2^n}(\mathfrak{A})$ is called the $2^n$-rank of $\mathfrak{A}$. Therefore Theorem \ref{main-thm-1}
shows that $r_8(K_2\CO_F)$ is always one. As far as we know, it is difficult to determine the $8$-rank for $K_2\CO_F$. In fact, our proof for Theorem \ref{main-thm-1} can not recover the result of Qin \cite{Q3} for $n=2$, in fact, the proof does depend on the result of Browkin and Schinzel \cite{BS} showing that $r_8(K_2\CO_F)\geq1$ when $n=2$. In addition to Theorem \ref{main-thm-1} for $r_8(K_2\CO_F)=1$,
we also  get the following theorem for $r_8(K_2\CO_F)=0$, in which we determine the $4$-rank for $K_2\CO_F$.

\bigskip

\begin{thm}\label{main-thm-2}
Let $F=\BQ(\sqrt{D})$ with discriminant $D>0$.
\begin{enumerate}
  \item Assume that $D=\ell_1\cdots \ell_n$ where $\ell_i$ are distinct primes such that $\ell_i\equiv 5\mod 8$ and $\left(\frac{\ell_i}{\ell_j}\right)=-1$ for $1\leq i<j\leq n$. Assume further that $n$ is odd. Then
\[K_2\CO_F(2)\simeq (\BZ/2\BZ)^{n+1}.\]
  \item Assume that $D=4\ell_1\cdots \ell_n$ satisfying the prime divisors $\ell_1\equiv 3\mod8$ and $\ell_i\equiv 5\mod 8\,  (2\leq i\leq n)$. Assume that $n$ is odd and $\left(\frac{\ell_i}{\ell_j}\right)=-1$ for $1\leq i<j\leq n$. Then
$K_2\CO_F(2)$ is isomorphic to $(\BZ/2\BZ)^{n-1}\times (\BZ/4\BZ)$ (resp. $(\BZ/2\BZ)^2$) if $n\geq3$ (resp. $n=1$).
\end{enumerate}
\end{thm}

For $n=1$, the above theorem is due to Browkin and Schinzel \cite[Corollary 1]{BS} (see also \cite{Q3}). Fortunately our proof for Theorem \ref{main-thm-2} recovers the case of Browkin and Schinzel. We also remark that the Theorem \ref{main-thm-2} is essentially used in the proof of Theorem \ref{main-thm-1}.

\medskip

Now, we say something about the main ideas in the proof for Theorems \ref{main-thm-1} and \ref{main-thm-2}. Previously, the $2^n$-rank $r_{2^n}(K_2\CO_F)$ was determined in an arithmetic way, i.e., using Hilbert symbols or solving certain Diophantine Equations in integers (see \cite{BS},\cite{Q1},\cite{Q2},\cite{Q3},\cite{YF}, \cite{CO} and \cite{Y1}). It may be possible that the method in \cite{Q3}, \cite{Q2} may show Theorems \ref{main-thm-1} and \ref{main-thm-2}. However, we stress that our method showing Theorems \ref{main-thm-1} and \ref{main-thm-2}
is analytic. The main tool is the Birch-Tate formula (see Theorem \ref{BT-formula}), which relates the order of $K_2\CO_F$ to $L(\chi_F,-1)$ (see Corollary \ref{k2-zeta}). Here, $\chi_F$ is the Dirichlet character associated to $F/\BQ$ and $L(\chi_F,s)$ is the Dirichlet L-functions associated to $\chi_F$. The Birch-Tate formula also shows that $L(\chi_F,-1)$ is a non-zero rational number. Thus if one can get a good estimate on the $2$-adic valuation for $L(\chi_F,-1)$, then one sees a rough picture of $K_2\CO_F(2)$. Although there are well-known special value formulas for $L(\chi_F,-1)$, it is very difficult to get enough information on the $2$-adic valuation for $L(\chi_F,-1)$ when the number of prime divisors of $D$ becomes large. To get around this difficulty, we use
an induction method for imprimitive Dirichlet L-functions. This method is inspired by Zhao who proved some cases for the Birch and Swinnerton-Dyer conjecture for elliptic curves with complex multiplication. By using this induction argument, we can get a more precise estimate on the $2$-adic valuation for $L(\chi_F,-1)$ when $D$ has an arbitrarily number of prime factors. Then by combining a well-known method for determining the $2$-rank and $4$-rank for ideal class groups and $K$ groups, we prove both Theorems. In the last step, we use a relation (due to Feng and Yue \cite{YF}) between the $4$-ranks of ideal class groups and $K$ groups.

\medskip

In the final part of the introduction, we briefly outline the contents of the paper. In next section \S\ref{btf}, we will state the Birch-Tate formula and give a precise relation between $K_2\CO_F$ and $L(\chi_F,-1)$. The main part of the paper is \S\ref{zhao-ch}, where we prove a special value formula for imprimitive Dirichlet L-functions. Then by using the induction argument, we get the required estimate on the $2$-adic valuation for the Dirichlet L-values (see Theorem \ref{mod-8-3}). In the final section \S\ref{k2-str}, we summarise the method to compute the $2$\, (resp. $4$)-rank of both the ideal class groups and $K$ groups. We then prove both theorems and give some numerical data in the last section.

\noindent\textbf{Acknowledgement}\quad We would like to thank Professors C. Greither, P. A. {\O}stv{\ae}r, H. Qin for helpful email correspondences, and thank Y. Kezuka for comments.

\section{Birch-Tate formula}\label{btf}

Let $\CK$ be a number field and $\CO_\CK$ the ring of integers in $\CK$. Let $\zeta_\CK(s)$ be the Dedekind zeta function of $\CK$. For any ring $R$, we let $K_2R$ denote the Milnor group of $R$ (equivalently, the second K group of $R$). By a theorem of Garland \cite{garland}, $K_2\CO_\CK$ is a finite abelian group.
In this section, for a real quadratic number field $F$,
we will state the Birch-Tate formula, which gives
a precise relation between the order of $K_2\CO_F$
and $\zeta_F(-1)$. Therefore, we can reduce the study of the structure of $K_2\CO_F$ to the zeta values.

\medskip

Let $\mathcal{F}$ be a totally real number field. The original Birch-Tate conjecture (see \cite{Tate69} for history and \cite{OCJM} for an update) predicts the following equality
\[\zeta_\mathcal{F}(-1)=\frac{\#(K_2(\CO_\mathcal{F}))}{\#(W_2(\mathcal{F}))},\]
where $W_2(\mathcal{F})=H^0(\mathcal{F},\BQ/\BZ(2))$. By Wiles's theorem on Iwasawa main conjecture (\cite{Wiles91}, see also \cite{greither} for abelian number fields) and a remark made by Kolster (see \cite{Weibel}), the above equality holds up to a power of $2$. Furthermore, if $\mathcal{F}$ is an abelian extension over $\BQ$,
the equality holds. Therefore, for a real quadratic number field $F$, we have the following Birch-Tate formula.

\begin{thm}\label{BT-formula}
Keep the above notations, then we have
\begin{equation}\label{BT-1}
  \zeta_F(-1)=\frac{\#(K_2(\CO_F))}{\#(W_2(F))}.
\end{equation}
\end{thm}

\medskip

The term $w_2(F)=\#(W_2(F))$ is easy to compute, we list the result in the following lemma.

\begin{lem}\label{w2}
Write $F=\BQ(\sqrt{D})$ with $D$ the discriminant of $F$. Then
\begin{enumerate}
  \item [(1)]If $D=5$, $w_2(F)=120$.
  \item [(2)]If $D=8$, $w_2(F)=48$.
  \item [(3)]If $D>8$, $w_2(F)=24$.
\end{enumerate}
 \end{lem}
\begin{proof}
This is a well-known results. We include a proof for readers' convenience. From \cite{w2}, we have the formula
\[w_2(F)=4\cdot\prod_{[\mathfrak{L}:F]=2}\frac{1}{2}\cdot\#(\mu(\mathfrak{L}))\]
where $\mathfrak{L}/F$ runs over all quadratic extensions of $F$, and $\mu(\mathfrak{L})$ is the group of roots of unity in $\mathfrak{L}$. By noting that a field extension over $\BQ$ with degree $4$ can only contain $4,5,6,8$-th roots of unity, the assertion in the lemma follows.
\end{proof}

When $F=\BQ(\sqrt{2})$ or $\BQ(\sqrt{5})$, one can easily determine the structure of $K_2\CO_F$ (at least for the $2$-primary part). From now on, we assume that $D>8$.

\begin{cor}\label{k2-zeta}
For $F=\BQ(\sqrt{D})$ satisfying $D>8$, we have
\[\#(K_2\CO_F)=24\cdot\zeta_F(-1).\]
In particular, if we denote by $\chi_F$ the Dirichlet character associated to $F/\BQ$, then
\[\#(K_2\CO_F)=-2\cdot L(\chi_F,-1).\]
Here, $L(\chi_F,s)$ is the Dirichlet L-series associated to $\chi_F$.
\end{cor}

\begin{proof}
Denote by $\zeta(s)$ the Riemann's zeta function.
The corollary follows by the well-known formula
$\zeta_F(s)=\zeta(s)L(s,\chi_F)$
and $\zeta(-1)=-\frac{1}{12}$.
\end{proof}

\section{A special value formula and induction method}\label{zhao-ch}

We know from section \S\ref{btf} that one can reduce the study of $K_2\CO_F$ to $L(\chi_F,-1)$.
It is also well-known that $L(\chi_F,-1)$ is a nonzero rational number
and there is a special value formula for $L(\chi_F,-1)$ in terms of the generalized Bernoulli numbers.
However, for any prime $\ell$, to determine the precise $\ell$-adic valuation of $L(\chi_F,-1)$ is very difficult (even for $\ell=2$). To overcome this difficulty, we borrow some ideas of Zhao, who proved some cases for the Birch and Swinnerton-Dyer conjecture for elliptic curves with complex multiplication (see \cite{Zhao}). The ideas consist of two main ingredients: one is a special value formula for imprimitive L-values of Dirichlet characters at $s=-1$; the other is an induction method for these imprimitive L-functions. By using Zhao's method, we can get a more precise estimate on the $2$-adic valuation of $L(\chi_F,-1)$ when $D$ has arbitrarily finitely many prime factors. This key result will help us to determine the $2$-primary part of $K_2\CO_F$ in the next section.

\medskip

Let us first recall some basic facts on Dirichlet L-functions and generalized Bernoulli numbers. Let $\chi$ be a Dirichlet character with conductor $f_\chi=f$. Let $g$ be any positive integer such that $f$ divides $g$. We define $L^{(g)}(\chi,s)(s\in\BC)$ to be the analytic continuation of the following power series
\[\sum^\infty_{n=1,\,  (n,g)=1}\frac{\chi(n)}{n^s},\qquad \textrm{where } \Re(s)>1.\]
This is an imprimitive Dirichlet L-series when $f\neq g$. When $f=g$, $L^{(g)}(\chi,s)$ is a primitive L-series and we simply denote it by $L(\chi,s)$.

We define the generalized Bernoulli number $B_{n,\chi,g}$ via the following identity
\begin{equation}\label{bn-g}
  \sum^g_{z=1}\frac{\chi(z)te^{zt}}{e^{gt}-1}=\sum^\infty_{n=0}B_{n,\chi,g}\frac{t^n}{n!}.
\end{equation}
For each integer $n\geq0$, we recall the definition of the Bernoulli polynomial $B_n(X)$ in the following equality
\begin{equation}\label{b-poly}
\frac{te^{Xt}}{e^t-1}=\sum^\infty_{n=0}B_n(X)\frac{t^n}{n!}.
\end{equation}
We remark that the dependence of $B_{n,\chi,g}$ on $g$ in \eqref{bn-g} can be ignored by the following lemma.

\begin{lem}
For any integer $g$ such that $f\mid g$, the number $B_{n,\chi,g}$ is independent of $g$. We denote this number by $B_{n,\chi}$. Then
\[B_{n,\chi}=g^{n-1}\sum^g_{a=1}\chi(a)\cdot B_n\left(\frac{a}{g}\right).\]
Furthermore, if $\chi$ is non-trivial and $\chi(-1)=1$, we have
\begin{equation}\label{b2-chi}
  B_{2,\chi}=\frac{1}{g}\sum^g_{a=1}\chi(a)a^2.
\end{equation}
\end{lem}

\begin{proof}
Set $y=\frac{g}{f}$ and make a change of variables
\[z=b+cf,\qquad 1\leq b\leq f,\, 0\leq c\leq y-1.\]
Noting that $\chi$ has conductor $f$, we have
\begin{equation}\label{3-1}
  \sum^g_{z=1}\chi(z)\cdot\frac{te^{zt}}{e^{gt}-1}=\sum^f_{b=1}\sum^{y-1}_{c=0}\chi(b)\cdot\frac{te^{t(b+cf)}}{e^{gt}-1}.
\end{equation}
Since the following identity
\[\sum^{y-1}_{c=0}e^{tcf}=\frac{e^{tg}-1}{e^{tf}-1},\]
the right hand side of \eqref{3-1} turns to be
\[\sum^f_{b=1}\sum^{y-1}_{c=0}\chi(b)\cdot\frac{te^{t(b+cf)}}{e^{gt}-1}=\sum^f_{b=1}\chi(b)\cdot\frac{te^{bt}}{e^{ft}-1}.\]
Therefore, we get
\[ \sum^g_{z=1}\chi(z)\cdot\frac{te^{zt}}{e^{gt}-1}=\sum^f_{b=1}\chi(b)\cdot\frac{te^{bt}}{e^{ft}-1}\]
and the first assertion follows. Next, by formula \eqref{b-poly}, we get
\[\sum^\infty_{n=0}g^{n-1}\sum^g_{a=1}\chi(a)B_n\left(\frac{a}{g}\right)\cdot\frac{t^n}{n!}=\sum^g_{a=1}\chi(a)\frac{te^{at}}{e^{gt}-1}.\]
Comparing to \eqref{bn-g}, the first equality in the lemma follows. Finally, note that $B_2(X)=X^2-X+\frac{1}{6}$, then formula \eqref{b2-chi} follows from the first equality. Here, we use the fact that $\chi$ is non-trivial and $\chi(-1)=1$. Therefore, the lemma follows.
\end{proof}

For any integer $n\geq1$,
the following special value formula for $L(\chi,s)$ at $s=1-n$ is well-known
\begin{equation}\label{3-2}
  L(\chi,1-n)=-\frac{B_{n,\chi}}{n}.
\end{equation}
One can find a detailed proof in \cite[Theorem 4.2]{Washington}.

\medskip

Next, we give a special value formula for imprimitive Dirichlet L-functions. For this purpose, we introduce some notation. For a positive integer $M\equiv 0$ or $1\mod 4$, we put $\chi_M=\left(\frac{M}{\cdot}\right)$, where $\left(\frac{M}{\cdot}\right)$ is the Jacobian symbol. Here, we make the convention that $\left(\frac{M}{-1}\right)=1$ since $M>0$ (see also \cite{shimura71}). Let $d_i(1\leq i\leq n)$ be discriminant of some real quadratic fields such that the greatest common divisor of $d_i,d_j (i\neq j)$ is $1$ or $4$.
For any $d=d_{i_1}\cdots d_{i_s}$ where $\{i_1,\cdots,i_s\}$ is a non-empty subset of $\{1,\cdots,n\}$, we let $d'=D/d$.

\begin{prop}\label{im-s-f}
Let $L^{(D)}(\chi_d,s)$ be the imprimitive Dirichlet L-function associated to $\chi_d$. Then
\begin{equation}\label{im-f}
  L^{(D)}(\chi_d,-1)=-\frac{1}{2D}\sum^{D}_{a=1}\chi_{Dd'}(a)a^2.
\end{equation}
\end{prop}

\begin{proof}
By the Euler product of $L^{(D)}(\chi_d,s)$ and that $\chi_d(\ell)=\left(\frac{d}{\ell}\right)$ for any prime $\ell\nmid d$, we get
\[L^{(D)}(\chi_d,s)=\prod_{\ell\nmid d'}\left(1-\left(\frac{d}{\ell}\right)\ell^{-s}\right)^{-1},\qquad \textrm{where}\quad \Re(s)>1.\]
Since for these $\ell\nmid d'$, we always have
$\chi_{Dd'}(\ell)=\chi_d(\ell)$. Therefore, the above equality implies that
\[L^{(D)}(\chi_d,s)=L(\chi_{Dd'},s).\]
Here, we should note that the conductor of $\chi_{Dd'}$ is $D$. Now, our formula \eqref{im-f} follows from the above equality, formulas \eqref{3-2} and \eqref{b2-chi}. Therefore, the Proposition follows.
\end{proof}

The next example are the very few cases where we can get a precise $2$-adic valuation for $L(\chi_F,-1)$ only using formula \eqref{im-f}.
We denote by $v=v_2$ the normalized additive valuation on $\BQ_2$ such that $v(2)=1$.

\begin{exm}\label{ex-bs}
Let $p\equiv 3\mod 8$ be a prime and $d=4p$. For $F=\BQ(\sqrt{d})$, we have
\[v(L(\chi_d,-1))=1.\]
Our motivation for this example comes from some arithmetic results by Browkin-Schinzel in \cite{BS}.
Of course, the discriminant $D_F=d$ of $F$ is greater than $8$, thus, from Birch-Tate formula, we know that $24\zeta_F(-1)\in\BZ$. By $\zeta(-1)=\frac{1}{12}$, formula \eqref{im-f} and
\[\zeta_F(-1)=\zeta(-1)L(\chi_d,-1),\]
we see that $\frac{1}{d}\sum^d_{a=1}\chi_d(a)a^2$ is an integer. By $\chi_d$ takes zero value on even integers and recall that $d=4p$, we get
\begin{equation}\label{3-3}
  \frac{1}{d}\sum^d_{a=1}\chi(a)a^2=\frac{1}{4p}\sum^{2p}_{b=1}\chi_d(2b-1)\cdot (2b-1)^2.
\end{equation}
Since $\chi_d$ is nontrivial and $\chi_d(-1)=1$, one can easily check that
\[\sum^{2p}_{b=1}\chi_d(2b-1)=\sum^{2p}_{b=1}\chi_d(2b-1)b=0.\]
Thus, from formula \eqref{3-3}, we get
\begin{equation}\label{3-4}
  \frac{1}{d}\sum^d_{a=1}\chi(a)a^2=\frac{1}{p}\sum^{2p}_{b=1}\chi_d(2b-1)\cdot b^2.
\end{equation}
Since $\chi_d(2b-1)=\left(\frac{p}{2b-1}\right)$, if we break up the second sum in \eqref{3-4} according as $b=2c$ or $b=2c-1$ ($1\leq c\leq p$), by modulo $8$ and noting that
\[(2c-1)^2\equiv 1\mod 8 \quad\textrm{and }\quad\sum^p_{c=1}\left(\frac{p}{4c-3}\right)=0,\]
then we get
\[ \frac{1}{d}\sum^d_{a=1}\chi(a)a^2\equiv \frac{1}{p}\sum^p_{c=1}\left(\frac{p}{4c-1}\right)4c^2\mod 8. \]
Now, the example will follow from the formula \eqref{im-f} and the following claim
\[\sum^p_{c=1}\left(\frac{p}{4c-1}\right)c^2\equiv 1\mod 2.\]
To show this claim, since the Jacobian symbol takes values in $\pm1$, thus, by modulo $2$, we have
\[\sum^p_{c=1}\left(\frac{p}{4c-1}\right)c^2\equiv \sum^p_{c=1}c^2 -\left(\frac{p+1}{4}\right)^2\mod 2.\]
Then one can easily check the right hand side in the congruence is indeed an odd integer, and our claim follows.
\end{exm}

\medskip

Recall that $D=d_1\cdots d_n$ with each $d_i$ is a discriminant of some real quadratic fields satisfying $(d_i,d_j)=1$ or $4$ for $i\neq j$. We remark that when $n$ is large, it is hardly to get a precise $2$-adic valuation of $L(\chi_D,-1)$, so now we introduce the induction method to get around this problem. Let us define
\[\mathcal{G}=\{a\in\BZ: 1\leq a\leq D,\, \chi_{d_i}(a)=1 \, \textrm{ for any } 1\leq i\leq n\}.\]
The heart of the induction method is to take sum of all the imprimitive L-values $L^{(D)}(\chi_d,-1)$, where $d=d_{i_1}\cdots d_{i_s}$ such that $\{i_1,\cdots,i_s\}$ runs over all non-empty subset of $\{1,\cdots,n\}$.

\begin{prop}\label{zhao1}
We have
\begin{equation}\label{zhao-f}
 \sum_{d=d_{i_1}\cdots d_{i_s}}L^{(D)}(\chi_{d},-1)=-\frac{2^{n-1}}{D}\sum_{1\leq a\leq D,\, a\in\CG}a^2+\frac{D}{6}\varphi(D)+\frac{1}{12}\prod_{p\mid D}(1-p).
\end{equation}
Here, the first sum is taken over all $d_{i_1}\cdots d_{i_s}$ such that $\{i_1,\cdots,i_s\}$ runs over all non-empty subset of $\{1,\cdots,n\}$. The last product is taken over all prime divisors $p$ of $D$. The function $\varphi(D)$ is the Euler function.
\end{prop}

\medskip

Before proving the Proposition, we need some elementary facts on Mobius function $\mu$.

\begin{lem}\label{mobius}
Assume $D$ as above and $D>1$, then we have
\[\sum_{m\mid D}\mu(m)=0,\qquad \sum_{m\mid D}\mu(m)\frac{D}{m}=\varphi(D),\qquad\textrm{and}\qquad \sum_{m\mid D}\mu(m)m=\prod_{p\mid D}(1-p).\]
\end{lem}
\begin{proof}
All are standard facts on Mobius functions, for a detailed proof, see \cite{A}.
\end{proof}

Now, we come back to the proof of Proposition \ref{zhao1}.

\begin{proof}[Proof of Proposition \ref{zhao1}]
Writing $d=d_{i_1}\cdots d_{i_s}$ and adding up the formulas \eqref{im-f} for all such $d\neq1$, we get
\[\sum_{d\neq 1}L^{(D)}(\chi_d,-1)=-\frac{1}{2D}\sum^D_{a=1}\chi_D(a)a^2\sum_{d'\neq D}\left(\frac{d'}{a}\right).\]
Here, we recall that $d'=\frac{D}{d}$ and $\chi_{Dd'}(a)=\chi_D(a)\left(\frac{d'}{a}\right)$.
By the following identity
\[\sum_{d'\mid D}\left(\frac{d'}{a}\right)=\prod^n_{i=1}\left(1+\left(\frac{d_i}{a}\right)\right),\]
we see that
\[\sum_{d\neq 1}L^{(D)}(\chi_d,-1)=-\frac{1}{2D}\sum^D_{a=1}\chi_D(a)a^2\left(\prod^n_{i=1}\left(1+\left(\frac{d_i}{a}\right)\right)-\left(\frac{D}{a}\right)\right).\]
For the term on the right hand of the above formula, by separating the bracket, we know the first term is non-zero if and only if $a\in\CG$. Noting that $\chi_D(a)=\left(\frac{D}{a}\right)$, then we get
\begin{equation}\label{ind-f-1}
  \sum_{d\neq 1}L^{(D)}(\chi_d,-1)=-\frac{2^{n-1}}{D}\sum_{1\leq a\leq D,\, a\in\CG}a^2+\frac{1}{2D}\sum_{1\leq a\leq D,\, (a,D)=1}a^2.
\end{equation}
To handle the second sum on right hand side in \eqref{ind-f-1}, we invoke some relation on arithmetic function on $\BZ$. For two such arithmetic functions $h_1,h_2$, their Dirichlet product is given as follows
\begin{equation}\label{ind-f-2}
  (h_1*h_2)(n)=\sum_{m\mid n}h_1(m)h_2(n/m).
\end{equation}
For the two arithmetic functions $\mathfrak{h}_1(n)=\sum_{1\leq a\leq n, \, (a,n)=1}\frac{a^2}{n^2}$ and $\mathfrak{h}(n)=\sum_{1\leq a\leq n}\frac{a^2}{n^2}$, one can easily show that $\mathfrak{h}_1=\mu*\mathfrak{h}$ (see also \cite[Exercise 2.14]{A}). Now, write the second sum in right hand side of \eqref{ind-f-1} as
\[\frac{1}{2D}\sum_{1\leq a\leq D,\, (a,D)=1}a^2=\frac{D}{2}\sum_{1\leq a\leq D,\, (a,D)=1}\frac{a^2}{D^2}.\]
Applying formula \eqref{ind-f-2} and Lemma \ref{mobius}, one can directly check the following equality holds
\[\frac{D}{2}\sum_{1\leq a\leq D,\, (a,D)=1}\frac{a^2}{D^2}=\frac{1}{12}\sum_{m\mid D}\mu(D/m)\cdot(m+1)(2m+1)\cdot\frac{D}{m}.\]
By using Lemma \ref{mobius} again, we see that
\[\frac{1}{12}\sum_{m\mid D}\mu(D/m)\cdot(m+1)(2m+1)\cdot\frac{D}{m}=\frac{D}{6}\varphi(D)+\frac{1}{12}\prod_{p\mid D}(1-p).\]
The Proposition now follows.
\end{proof}

\medskip

In the final part of this section, we apply the formula \eqref{zhao-f} to some special $d_i$'s to get some useful estimation on the $2$-adic valuation for $L(\chi_D,-1)$.

\begin{prop}\label{mod-8-5}
Let $D=d_1\cdots d_n$ be given as before. Assume that each $d_i=p_i\equiv 5\mod 8$ are primes, then we have
\[v(L(\chi_D,-1))\geq n.\]
Moreover, if we assume that $\left(\frac{p_i}{p_j}\right)=-1$ for all $1\leq i<j\leq n$, then
\begin{enumerate}
  \item [(1)]If $n$ is even,
  $v(L(\chi_D,-1))\geq n+1$.
  \item [(2)]If $n$ is odd,
  $v(L(\chi_D,-1))=n$.
\end{enumerate}
\end{prop}

\begin{proof}
Let us first prove the general lower bound in the Proposition. The main idea is by induction on $n$ using formula \eqref{zhao-f}. The case for $n=1$ is similar to Example \ref{ex-bs}. We simply write $d_1=p$ and assume that $p>5$. Recall that
\begin{equation}\label{xp-f}
  L(\chi_p,-1)=-\frac{1}{2p}\sum^p_{a=1}\left(\frac{p}{a}\right) a^2.
\end{equation}
For the sum $\sum^p_{a=1}\left(\frac{p}{a}\right) a^2$, separating it according as $a=2b$ or $2b-1$ with $1\leq b\leq\frac{p-1}{2}$ and noting the following fact (since $p\equiv 5\mod 8$)
\[\sum^{\frac{p-1}{2}}_{b=1}\left(\frac{p}{2b-1}\right)=\frac{1}{2}\sum^{\frac{p-1}{2}}_{b=1}\left(\left(\frac{p}{2b-1}\right)+\left(\frac{p}{p-(2b-1)}\right)\right)=0,\]
we get
\[\sum^p_{a=1}\left(\frac{p}{a}\right) a^2=4\left(\sum^{\frac{p-1}{2}}_{b=1}\left(\frac{p}{2b}\right)b^2+\sum^{\frac{p-1}{2}}_{b=1}(b^2-b)\left(\frac{p}{2b-1}\right)\right)=:4\mathfrak{s}.\]
We then claim that $\mathfrak{s}$ is odd. In fact, by modulo $2$ and recalling that $p\equiv 5\mod 8$, we see that
\[\mathfrak{s}\equiv \sum^{\frac{p-1}{2}}_{b=1} b\equiv\frac{(p+1)(p-1)}{8} \equiv1\mod 2.\]
Therefore, by formula \eqref{xp-f}, we have $v(L(\chi_p,-1))=1$. The case for $n=1$ then follows. Now, we let $n\geq2$ and assume the proposition holds when $d\neq D$, i.e. $v(L(\chi_d,-1))\geq s$. From formula \eqref{zhao-f}, we get the following key formula in our induction argument
\begin{equation}\label{key-id}
 \sum_{d\neq 1}L^{(D)}(\chi_d,-1)=-\frac{2^{n-1}}{D}\sum_{1\leq a\leq D,\, a\in\CG}a^2+\left(\frac{D}{6}+\frac{(-1)^n}{12}\right)\varphi(D).
\end{equation}
Here, the first sum is taken with all $d=d_{i_1}\cdots d_{i_s}$ dividing $D$ such that $s>0$. For the term $\sum_{1\leq a\leq D,\, a\in\CG}a^2$, by the lemma below, it has $2$-adic valuation $1$ or at least $2$ according as $n=2$ or $n\geq4$. In either cases, we always have
\[v\left(-\frac{2^{n-1}}{D}\sum_{1\leq a\leq D,\, a\in\CG}a^2\right)\geq n.\]
It is also easily to see that $v\left(\left(\frac{D}{6}+\frac{(-1)^n}{12}\right)\varphi(D)\right)=2n-2$, which is at least $n$ since $n\geq2$. Next, we deal with each term $L^{(D)}(\chi_d,-1)$ in the left hand side in \eqref{key-id}. Note the equality
\[L^{(D)}(\chi_d,-1)=\left(\prod_{p\nmid d'}(1-\chi_d(p)p)\right)\cdot L(\chi_d,-1)\qquad d'=D/d.\]
The induction hypothesis shows that if $d\neq D$, $v(L(\chi_d,-1))\geq s$. Noting that $v(1-\chi_d(p)p)\geq1$ for all $p\mid d'$, we get
\[v(L^{(D)}(\chi_d,-1))=v\left(\left(\prod_{p\nmid d'}(1-\chi_d(p)p)\right)\right)+v(L(\chi_d,-1))\geq n\]
for all $d\neq D$. Finally, by formula \eqref{key-id}, we must have
\[v(L(\chi_D,-1))\geq n,\]
thus
the first assertion of the proposition follows. For the next two special cases, since the method are similar, we only give a detailed proof for (2). We assume now that $\left(\frac{p_i}{p_j}\right)=-1$ for $i\neq j$ and $n$ is odd. We still use induction on $n$, the case $n=1$ is done in the above. Write $d=d_{i_1}\cdots d_{i_s}$. If $s$ is even, by $n$ is odd, we have $s+1\leq n$. Then
$L^{(D)}(\chi_d,-1)=\left(\prod_{p\mid d'}(1-p)\right)L(\chi_d,-1)$ has $2$-adic valuation at least $2n-s\geq n+1$. Since $n\geq3$, by the lemma below and a similar method as above, each term in the right hand side in \eqref{key-id} has $2$-adic valuation at least $n+1$. Now, for the terms $L^{(D)}(\chi_d,-1)$ with $s<n$ odd, by induction hypothesis, we have $v(L^{(D)}(\chi_d,-1))=n$. Note that the number of all such terms are odd, therefore, from \eqref{key-id}, we must have $v(L(\chi_D,-1))=n$, (1) then follows.
\end{proof}

We now prove the following lemma used in the proof of Proposition \ref{mod-8-5}.

\begin{lem}\label{a2-sum1}
Recall that $D=d_1\cdots d_n$ with $d_i=p_i\equiv 5\mod 8$ are primes. For $\mathfrak{s}_1=\sum_{1\leq a\leq D,\, a\in\mathcal{G}}a^2$, we have
\[\mathfrak{s}_1\equiv 2^{n-1}\mathfrak{t}\mod 4.\]
where $\mathfrak{t}$ is some odd integer. In particular, when $n\geq3$, $\mathfrak{s}_1$ is always divisible by $4$.
\end{lem}

\begin{proof}
For $i=0,1$, we define $\mathfrak{G}_i$ to be the set
$\{a\in \mathcal{G}: a\equiv i\mod 2\}$. Since sending $x\in \mathfrak{G}_0$ to $D-x\in \mathfrak{G}_1$ is a bijection and that it is easy to show that $\#(\mathfrak{G}_1\cup\mathfrak{G}_2)=\frac{\varphi(D)}{2^n}$, then we have
\[\#(\mathfrak{G}_1)=\frac{1}{2}\prod^n_{i=1}\frac{p_i-1}{2}.\]
Now, the lemma follows from the following congruence
\[\mathfrak{s}_1\equiv \#(\mathfrak{G}_1)\mod 4.\]

\end{proof}

\bigskip

The following theorem is the main result in this section.

\begin{thm}\label{mod-8-3}
Let $D=d_1\cdots d_n$ satisfying $d_1=4p_1$ where $p_1\equiv 3\mod 8$ is a prime and $d_i=p_i\equiv 5\mod 8$ are primes for $i>1$. Then we have
\begin{equation}\label{8-3-f1}
  v(L(\chi_D,-1))\geq n.
\end{equation}
Furthermore,
if we
assume that $\left(\frac{p_i}{p_j}\right)=-1$ for all $2\leq i<j\leq n$.
\begin{enumerate}
  \item [(1)]If $\left(\frac{p_1}{p_j}\right)=-1$ for $j>1$ and $n$ is odd, then $v(L(\chi_D,-1))=n$.
  \item [(2)]Assume that one of the following conditions holds
  \begin{enumerate}
    \item [(2a)]$\left(\frac{p_1}{p_j}\right)=-1$ for $j\geq2$ and $n$ is even.
    \item [(2b)]$\left(\frac{p_1}{p_j}\right)=1$ for $j\geq2$.
    \item [(2c)]$\left(\frac{p_1}{p_2}\right)=-1$, $\left(\frac{p_1}{p_j}\right)=1$ for $j>2$ and $n$ is even.
  \end{enumerate}
  Then $v(L(\chi_D,-1))\geq n+1$.
\end{enumerate}
\end{thm}

\bigskip

Here, (2c) will be our concerned case to study the structure of the $2$-primary part of $K_2\CO_F$ in section \S\ref{k2-str}. Before giving the proof, we must say that the proof of (2c) involves all the before cases, i.e., Proposition \ref{mod-8-5}, (1) and (2a) to (2b) in Proposition \ref{mod-8-3} in the induction arguments and also a key arithmetic result of Browkin and Schinzel \cite{BS} (see also \cite{Q3}), which we state in the following lemma.

\begin{lem}[Browkin and Schinzel]\label{bs-lem}
Let $D=d_1d_2$ satisfying $d_1=4p_1$ where $p_1\equiv 3\mod 8$ is a prime and $d_2=p_2\equiv 5\mod 8$ is a prime. Then we have
\[v(L(\chi_D,-1))\geq3.\]
\end{lem}

\begin{proof}
Write $F=\BQ(\sqrt{D})$.
Recall that $K_2\CO_F(2)$ denotes the $2$-primary part of $K_2\CO_F$. By \cite[Corollary 2]{BS}, we have $\#(K_2\CO_F(2))\geq16$.
Then the lemma follows from Corollary \ref{k2-zeta}.
\end{proof}

\medskip

As usual the formula \eqref{zhao-f} in Proposition \ref{zhao1} in this case becomes
\begin{equation}\label{8-3-f2}
  \sum_{d\neq1} L^{(D)}(\chi_d,-1)=-\frac{2^{n-1}}{D}\sum_{1\leq a\leq D,\, a\in\CG}a^2+\left(\frac{D}{6}+\frac{(-1)^{n+1}}{24}\right)\varphi(D).
\end{equation}
By a more delicate analysis on elements in $\CG$, we can get the following key estimation for the $2$-adic valuation of the sum $\sum_{1\leq a\leq D,\, a\in\CG}a^2$.

\begin{lem}\label{a2-sum2}
Let $D$ be given as in Theorem \ref{mod-8-3}. Put $\mathfrak{t}_1=\sum_{1\leq a\leq D,\, a\in\CG}a^2$. Then
\[\mathfrak{t}_1\equiv 2^n \mathfrak{y} \mod 16.\]
Here, $\mathfrak{y}$ is some odd integer.
In particular, when $n\geq4$, we have
$v(\mathfrak{t}_1)\geq4$.
\end{lem}

\begin{proof}
Note that $D$ is even. Thus any $a\in \CG$ is an odd integer. For $i=0,1$, we put
\[\mathfrak{W}_i=\{a\in \CG: a^2\equiv 1+8i\mod 16\}.\]
One can easily see that $a\in \CG$ if and only if $D-a\in\CG$. We claim further that the automorphism of $\CG$ sending $a$
to $D-a$ maps $\mathfrak{W}_0$ onto $\mathfrak{W}_1$. In fact, for $i=0,1$, note that $a\equiv \pm(1+2i)\mod 8$ if and only if $a^2\equiv 1+8i\mod 16$, i.e., $a\in\mathfrak{W}_i$. By a direct checking on the residue class for $D-a$, the claim follows. From the claim, we know that $\#(\mathfrak{W}_1)=\#(\mathfrak{W}_2)$. Therefore, we get
\[\mathfrak{t}_1\equiv \#(\mathfrak{W}_0)+9\cdot\#(\mathfrak{W}_1)=10\cdot \#(\mathfrak{W}_0)\mod 16.\]
Since $\CG$ is a disjoint union of $\mathfrak{W}_0$ and $\mathfrak{W}_1$ and that $\CG$ has $\frac{\varphi(D)}{2^n}$ elements, then we have
\[\mathfrak{t}_1\equiv 10 \cdot \prod^n_{i=1}\left(\frac{p_i-1}{2}\right)\mod 16.\]
Therefore, put $\mathfrak{y}=\left(10 \cdot \prod^n_{i=1}\left(\frac{p_i-1}{2}\right)\right)/2^n$, the lemma follows.
\end{proof}

Now, we come back to the proof for Theorem \ref{mod-8-3}.

\begin{proof}[Proof of Theorem \ref{mod-8-3}]
We first show that inequality \eqref{8-3-f1} holds. When $n=1$, it has been done in Example \ref{ex-bs}. Using both the method and results in Proposition \ref{mod-8-5}, one can also show that \eqref{8-3-f1} holds for $n=2,3$. So, we may assume that $n\geq4$. The two terms on the right hand side in \eqref{8-3-f2} have $2$-adic valuations given as follows
\[v\left(\left(\frac{D}{6}+\frac{(-1)^{n+1}}{24}\right)\varphi(D)\right)=2n-3\geq n+1,\]
and
\[v\left(-\frac{2^{n-1}}{D}\sum_{1\leq a\leq D,\, a\in\CG}a^2\right)\geq n+1.\]
Here, we use the fact that $n\geq4$ in the first inequality and Lemma \ref{a2-sum2} in the second. For each term with $d\neq D$ on the left hand side in \eqref{8-3-f2}, by Euler product for $L^{(D)}(\chi_d,-1)$ and the following facts
\[v(1-\chi_d(2)2)=0\quad\textrm{  and  }\quad v(1-\chi_d(p)p)\geq1\]
for every odd prime $p\mid D$, noting also the induction hypothesis, we get
$$v(L(\chi_D,-1))\geq n.$$
Therefore, the inequality \eqref{8-3-f1} follows. Next, we prove (1). For $n=1$, it is done in the above proof. One can also check directly that the case for $n=3$ is true. So we may assume that $n\geq5$. Then each term on the right hand side in \eqref{8-3-f2} has $2$-adic valuation at least $n+1$. Thus, we only need to deal with the terms $L^{(D)}(\chi_d,-1)$ (with $d\neq D$) on the left hand side of \eqref{8-3-f2}. Recall that $s$ is the positive integer such that $d=d_{i_1}\cdots d_{i_s}$. By Proposition \ref{mod-8-5}, formula \eqref{8-3-f1} and the induction hypothesis, we can show that $v(L^{(D)}(\chi_d,-1))\geq n+1$ holds for the following two cases:
either $s$ is even
or $s$ is odd and $d_1$ does not divide $d$. For the remaining case where $s$ is odd and $d_1\mid d$, since $\chi_d(p)=-1$ for every prime $p\mid d'=D/d$, thus $L^{(D)}(\chi_d,-1)=\prod_{p\mid d'}(1+p)\cdot L(\chi_d,-1)$. Since the induction hypothesis shows that $v(L(\chi_d,-1))=s$, noting that each $p\mid d'$ satisfies $p\equiv 5\mod 8$, we must have $v(L^{(D)}(\chi_d,-1))=n$. Since $n\geq5$, we see that the total number of $d\neq D$ in this case is odd, therefore, from formula \eqref{8-3-f2}, we get $v(L(\chi_D,-1))=n$. Therefore, (1) follows. To prove (2), by using the same proof for $n=1$ in \eqref{8-3-f1}, Lemma \ref{bs-lem} and a direct computation, we can prove that (2) holds for $n<4$. Therefore, we may assume that $n\geq4$, then the right hand side of \eqref{8-3-f2} has $2$-adic valuation at least $n+1$. So, we only need to deal with the terms $L^{(D)}(\chi_d,-1)$ for $d\neq D$.
Let us first prove (2a). If $s$ is even, $v(L^{(D)}(\chi_d,-1))\geq n+1$ follows from the induction hypothesis. If $s$ is odd, from (1) and $\chi_d(p)=-1$, we see that $v(L^{(D)}(\chi_d,-1))=n$. Note that the number of all $d\neq D$ such that $s$ is odd is an even integer, thus by \eqref{8-3-f2} we must have $v(L(\chi_D,-1))\geq n+1$, then (2a) follows. For the case (2b), from the induction hypothesis and Proposition \ref{8-3-f1}, we can easily check that $v(L^{(D)}(\chi_d,-1))\geq n+1$ holds for the cases: either $d_1\mid d$ or $d_1\nmid d$ and $s$ is even. For the remaining case when $d_1\nmid d$ and $s$ is odd, note that $p_1\mid d'$, and for each $p\mid d'$, $\chi_d(p)=1$ or $-1$ according as $p=p_1$ or $p\neq p_1$, we get
\[v(L^{(D)}(\chi_d,-1))=v((1-p_1)\cdot\prod_{p\mid d',\, p\neq p_1}(1+p)L(\chi_d,-1))=n.\]
Here we use (2) in Proposition \ref{mod-8-5} in the last equality. Since the number of all such $d$'s is even, thus the sum of all these $L^{(D)}(\chi_d,-1)$'s has $2$-adic valuation at least $n+1$.
Therefore, from equality \eqref{8-3-f2}, we have $v(L(\chi_D,-1))\geq n+1$. Therefore, (2b) follows. Finally, we prove (2c). Since $n$ is even, and from Lemma \ref{bs-lem}, we know that (2c) holds for $n=2$. Thus we may assume that $n\geq4$. So, the $2$-adic valuation on the right hand side of \eqref{8-3-f2} is at least $n+1$. To get the assertion in (2c), we only need to study the terms $L^{(D)}(\chi_d,-1)$ for $d\neq D$ on the left hand side of \eqref{8-3-f2}. We divide the proof into the following four cases:
\begin{enumerate}
   \item [Case (i)]If $d_1\mid d$ and $s$ is even. Then by either the induction hypothesis or (2b), we get $v(L(\chi_d,-1))\geq s+1$. Therefore, we must have $v(L^{{D}}(\chi_d,-1))\geq n+1$ for such $d$'s.
  \item [Case(ii)]If $d_1\mid d$ and $s$ is odd. For every $p\mid d'$, we see that $\chi_d(p)=-1$ (resp. $1$) if $p_2\mid d'$ and $p=p_2$ (resp. $p\neq p_2$). Then $v\left(\prod_{p\mid d'}(1-\chi_d(p)p)\right)$ is at least $2(n-s)-1$ (resp. $2(n-s)$) according as $p_2\mid d'$ (resp. $p_2\nmid d'$). If $p_2\mid d'$, so $p_2\nmid d$, by (2b), we have $v(L(\chi_d,-1))\geq s+1$. Thus we get $v(L^{(D)}(\chi_d,-1))\geq 2n-s\geq n+1$ since $n$ is even. If $p_2\nmid d'$, so $p_2\mid d$, by the inequality \eqref{8-3-f1}, we have $v(L(\chi_d,-1))\geq s$. Thus we get $v(L^{(D)}(\chi_d,-1))\geq 2n-s\geq n+1$. Therefore, case (ii) follows.
  \item [Case(iii)]If $d_1\nmid d$ and $s$ is even. From (1) in Proposition \ref{mod-8-5}, we always have $v(L(\chi_d,-1))\geq s+1$. Thus we get
  $v(L^{(D)}(\chi_d,-1))\geq n+1$.
  \item [Case(iv)]If $d_1\nmid d$ and $s$ is odd. From (2) in Proposition \ref{mod-8-5}, we have $v(L(\chi_d,-1))=s$. Since $p_1\mid d'$, for every $p\mid d'$, we can easily check that $\chi_d(p)=-1$ or $1$ according as $p_2\mid d$ or $p_2\nmid d$ and $p=p_1$. So, if $p_2\mid d$, we get $v\left(\prod_{p\mid d'}(1-\chi_d(p)p)\right)=n-s+1$, thus we must have $v(L^{(D)}(\chi_d,-1))\geq n+1$. For the remaining case when $p_2\nmid d$, note that $v\left(\prod_{p\mid d'}(1-\chi_d(p)p)\right)=n-s$, we get $v(L^{(D)}(\chi_d,-1))=n$. However, the total number of $d$'s such that $d_1p_2$ does not divide $d$ is even (since $n\geq4$), thus the sum of these $L^{(D)}(\chi_d,-1)$ must have $2$-adic valuation at least $n+1$.
\end{enumerate}
Now, by noting the equality \eqref{8-3-f2}, we have $v(L(\chi_D,-1))\geq n+1$, thus (2c) follows and the proof for the Theorem is done.
\end{proof}

\section{Structure for $2$-primary part of $K_2\CO_F$}\label{k2-str}

Let $K=\BQ(\sqrt{d})$ be a quadratic number field and denote by $\mathcal{C}_K$ the ideal class group of $K$. In this section, we first collect some standard results on $2$-rank and $4$-rank for $\mathcal{C}_K$ and $K_2\CO_K$. Then we give applications of results in section \S\ref{zhao-ch} to determine the $2$-primary structure of $K_2\CO_F$ for real quadratic number fields $F$.

\medskip

Let us denote by $D_K$ the discriminant of $K$ and $t$ the number of prime divisors of $D_K$. Let $\left(\frac{\cdot}{\ell}\right)$ be the Legendre symbol if $\ell$ is an odd prime and $\left(\frac{\cdot}{2}\right)$ the Kronecker symbol.
Let $p^*=\left(\frac{-1}{p}\right)p$ if $p$ is an odd prime and $2^*=\pm4$ or $\pm 8$ such that $D_K=\prod^t_{i=1}p^*_i$. Put $r_{ij}=\frac{1}{2}\left(1-\left(\frac{p^*_i}{p_j}\right)\right)$, where $1\leq i, j\leq t$ and $i\neq j$. Set $r_{kk}=\sum_{i\neq k}r_{ik}$.  We make the convention that $p_1=2$ if $2\mid D_K$.
We then define the Redei matrix $R_K$ associated to $K$ to be
$R_F=(r_{ij})$. Here, the matrix $R_F$ is viewed as a $t\times t$-matrix over the finite field $\BF_2$ with $2$-elements. Denote by $r(R_K)$ the rank of the matrix $R_F$ over $\BF_2$. Recall that the narrow class group $\mathcal{C}^+_K$ of $K$ is
the quotient of the group of fractional $\CO_K$-ideals by the subgroup of principal
ideals $(x)=x\CO_K$ with generator of positive norm $N(x)$. Here, we write $N=N_K$ for the norm map from $K$ to $\BQ$.

\begin{prop}\label{24-pic}
Keep the notations as above. Then
\begin{enumerate}
  \item [(1)]$r_2(\mathcal{C}^+_K)=t-1$.
  \item [(2)]If $K$ is a real quadratic field and $r_4(\mathcal{C}_K)=0$, then $r_2(\mathcal{C}_K)=t-1$\,  (resp. $t-2$) if $-1\in N(K^\times)$ \,  (resp. $-1\notin N(K^\times)$).
  \item [(3)]$r_4(\mathcal{C}^+_K)=t-1-r(R_K)$. In particular, if the $4$-rank of $\mathcal{C}^+_K$ is zero, then so is the ideal class group $\mathcal{C}_K$.
\end{enumerate}
\end{prop}
\begin{proof}
All the assertions are well-known. One can find a proof in \cite[Theorems (2.1) and (3.1)]{ste}. The in particular part in (3) can be shown in the following way. Write $\mathfrak{P}$ (resp. $\mathfrak{P}^+$) for the group generated by principal ideals (resp. principal ideals with positive norm generate). Then we have the following exact sequence
\[0\to \mathfrak{P}/\mathfrak{P^+}\to \mathcal{C}^+_K \to \mathcal{C}_K \to 0.\]
Since $\mathfrak{P}/\mathfrak{P^+}$ is annihilated by $2$, we see that the in particular part in (3) follows.
\end{proof}

We now write $F$ for a real quadratic number field with discriminant $D$ and put $E=\BQ(\sqrt{-D})$. Let $S$ be the set of primes in $F$ lying above $2$. We denote by $\CO_{F,S}$ the ring of $S$-integers in $F$ and put $\mathcal{C}_{F,S}$ the ideal class group of $\CO_{F,S}$.
We omit the proof for the below Proposition and one can find a proof in \cite{BS} and \cite{YF} (see also \cite{Q1}).

\begin{prop}\label{24-k2}
Keep the notions as above. Denote by $s=\#(S)$, then we have
\begin{enumerate}
  \item [(1)]
  $r_2(K_2\CO_F)=1+r_2(\mathcal{C}_{F,S})+s$.
  \item [(2)]
  $\left|r_4(K_2\CO_F)-r_4(\mathcal{C}_E)\right|\leq 1$.
\end{enumerate}
\end{prop}

\bigskip

Now for a finite abelian group $\mathfrak{A}$, we recall that $\mathfrak{A}(2)$ is the $2$-primary subgroup of $\mathfrak{A}$.
The following two propositions prove the Theorem \ref{main-thm-2} in the introduction.

\begin{prop}\label{k2-2}
Assume that $D=p_1\cdots p_n$ satisfying $p_i$ are primes such that $p_i\equiv 5\mod 8$ and $\left(\frac{p_i}{p_j}\right)=-1$ for $1\leq i<j\leq n$. If $n$ is odd, then we have
\[K_2\CO_F(2)\simeq (\BZ/2\BZ)^{n+1}.\]
\end{prop}

\begin{proof}
Since $n$ is odd, so $D\equiv 5\mod 8$. Then $2$ is inert in $F$. Therefore, $s=1$ and $\mathcal{C}_{F,S}=\mathcal{C}_F$. We claim that $-1\in N_F(F^\times)$, in fact, by modulo every $p\mid D$ and noting that $p\equiv 1\mod 4$, there exists integers $x,y,z$ not all zero satisfying
\[x^2-y^2D\equiv-z^2 \mod p.\]
Note also that $x^2-y^2D=-z^2$ has solutions in $\BR$. Thus by Legendre's theorem (see \cite[Lemma (2.1)]{Q2}), $x^2-y^2D=-z^2$ has nontrivial solutions in $\BZ$. Therefore, the claim follows.
By Proposition \ref{24-pic}, we get $r_2(\mathcal{C}_F)=n-1$. Note Proposition \ref{24-k2}, then  $r_2(K_2\CO_F)=n+1$. Now the Proposition follows from (2) in Proposition \ref{mod-8-5} and Corollary \ref{k2-zeta}.
\end{proof}

\medskip

\begin{prop}\label{k2-4}
Let $F=\BQ(\sqrt{D})$ with discriminant $D=4\ell_1\cdots \ell_n$ satisfying the prime divisors $\ell_1\equiv 3\mod8$ and $\ell_i\equiv 5\mod 8\,  (2\leq i\leq n)$. Assume that $n$ is odd and $\left(\frac{\ell_i}{\ell_j}\right)=-1$ for $1\leq i<j\leq n$. Then
Then
$K_2\CO_F(2)$ is isomorphic to $(\BZ/2\BZ)^{n-1}\times (\BZ/4\BZ)$ (resp. $(\BZ/2\BZ)^2$) if $n\geq3$ (resp. $n=1$).
\end{prop}

\begin{proof}
The case for $n=1$ is easy, we left to the reader. We now assume that $n\geq3$.
Since $\ell_1\equiv 3\mod 4$, by the same argument as in Proposition \ref{k2-2}, we have $-1\notin N(F^\times)$. Thus from Proposition \ref{24-pic}
we see that $r_2(\mathcal{C}_F)=n-1$. Now, writing $D=p_1^*\cdots p_t^*$ and recalling the convention that $p^*_1=-4$, we know the Redei matrix $R_F$ is given as the following block matrix
\[\left(
    \begin{array}{cc}
      A_2 & \alpha_1 \\
      \beta_1 & B_{n-1} \\
    \end{array}
  \right),
\]
where $A_2$ (resp. $\beta_1$) is a $2\times 2$ (resp. $(n-1)\times2$) matrix with all entries equal to $1$, $B_{n-1}$ is a $(n-1)\times (n-1)$ matrix with all non-diagonal entries equal to $1$ and all diagonal entries equal to $0$, and $\alpha_1$ is a $2\times(n-1)$ matrix with all elements in the first (resp. second) row equal to $0$ (resp. $1$). By applying an elementary row operation to $R_F$, we see that $r(R_F)=n$. Therefore, from Proposition \ref{24-pic}, we get $r_4(\mathcal{C}_F)=0$. Note that $2$ is ramified in $F$, so $s=1$. One can easily show that the unique prime $\mathfrak{p}$ in $F$ lying above $2$ is not principal. So the ideal class of $\mathfrak{p}$ generates a subgroup in $\mathcal{C}_F$ isomorphic to $\BZ/2\BZ$. Therefore, we have
$r_2(\mathcal{C}_{F,S})=n-2$. From Proposition \ref{24-k2}, we get $r_2(K_2\CO_F)=n$. To get the $4$-rank for $K_2\CO_F$, we must study the $4$-rank for $E=\BQ(\sqrt{-D})$. By Proposition \ref{24-pic}, $r_2(\mathcal{C}_E)=n-1$ and the Redei matrix $R_E$ is given by $B_n$. since $n$ is odd, we see that $r(B_n)=n-1$. Therefore, $r_4(\mathcal{C}_E)=0$ and from Proposition \ref{24-k2}, we only have $r_4(K_2\CO_F)\leq 1$. On the other hand, by (1) in Theorem \ref{mod-8-3} and Corollary \ref{k2-zeta}, we have
\[v(\#(K_2\CO_F))=n+1.\]
Then by combining $r_2(K_2\CO_F)=n$ and $r_4(K_2\CO_F)\leq 1$, we must have $r_4(K_2\CO_F)=1$, $r_8(K_2\CO_F)=0$ and
the proposition follows.
\end{proof}

\medskip

Now, we can prove the Theorem \ref{main-thm-1}. It is the application of (2c) in Proposition \ref{mod-8-3} mentioned in section \S\ref{zhao-ch}.

\begin{proof}[Proof of Theorem \ref{main-thm-1}]
The heart part of the proof is (2c) in Theorem \ref{mod-8-3}, which we have already proven. So to complete the proof, we only need to
show that $r_4(K_2\CO_F)\leq 1$ and $r_2(K_2\CO_F)=n$. In fact, from (2c) in Theorem \ref{mod-8-3} and Corollary \ref{k2-zeta}, then the isomorphism in the Theorem follows.

Since $\ell_1\equiv 3\mod 8$ we know that $-1\notin N(F^\times)$. Thus from Proposition \ref{24-pic}, we have $r_2(\mathcal{C}_F)=n-1$.
The Redei matrix $R_F$ is given as
\[R_F=\left(
        \begin{array}{cc}
          C & \gamma_1 \\
          \gamma_2 & B_{n-2} \\
        \end{array}
      \right)\qquad C\in M_3(\BF_2),\]
where $B_{n-2}$ is the same matrix given in Proposition \ref{k2-4}, and $C,\gamma_i(i=1,2)$ can be explicitly determined. Here, we should note that the sum of all elements in the each column is zero.
We can show that $B_{n-2}$ is of full rank, and the rank of $R_F$ is $n$. Therefore, from Proposition \ref{24-pic}, we have $r_4(\mathcal{C}_F)=0$. Note that the ramified prime $\mathfrak{p}$ of $F$ lying above $2$ is not principal. Therefore, $r_2(\mathcal{C}_{F,S})=n-2$ and $s=1$. So we have $r_2(K_2\CO_F)=n$ by Proposition \ref{24-k2}. To show that $r_4(K_2\CO_F)\leq 1$, we consider the field $E=\BQ(\sqrt{-D})$, an easy calculation using Proposition \ref{24-pic} shows that $r_2(\mathcal{C}_E)=n-1$, $r(R_E)=n-1$ and $r_4(\mathcal{C}_E)=0$. Thus, by Proposition \ref{24-k2} we must have $r_4(K_2\CO_E)\leq 1$. The theorem now follows.
\end{proof}

In the final part of this section, we give some numerical data for the case $n=4$ in the following table. The data shows that $\delta$ is always equal to $3$, however, at present we can not prove it theoretically.

\begin{center}
\begin{scriptsize}
\begin{table}[!htbp]
\center
\begin{tabular}{|c|c|c|c|c|c|c|}
    \hline \(\frac{D}{4}\)&\(p_1\)&\(p_2\)&\(p_3\)&\(p_4\)&\(L(\chi_D,-1)\)&\(\delta\)\\
    \hline 7215 & 3 & 5 & 13 & 37 & 240480 & 3\\
    \hline 26455 & 11 & 13 & 5 & 37 & 1997920 & 3\\
    \hline 77415 & 3 & 5 & 13 & 397 & 8824224 & 3\\
    \hline 119535 & 3 & 5 & 13 & 613 & 16692000 & 3\\
    \hline 142935 & 3 & 5 & 13 & 733 & 22829088 & 3\\
    \hline 153735 & 3 & 5 & 37 & 277 & 25086816 & 3\\
    \hline 166335 & 3 & 5 & 13 & 853 & 28304352 & 3\\
    \hline 171015 & 3 & 5 & 13 & 877 & 28115040 & 3\\
    \hline 196359 & 3 & 29 & 37 & 61 & 37791648 & 3\\
    \hline 226655 & 11 & 13 & 5 & 317 & 40287584 & 3\\
    \hline 241215 & 3 & 5 & 13 & 1237 & 48586080 & 3\\
    \hline 243295 & 19 & 13 & 5 & 197 & 53437792 & 3\\
    \hline 257335 & 107 & 5 & 13 & 37 & 60717792 & 3\\
    \hline 283855 & 11 & 13 & 5 & 397 & 67222496 & 3\\
    \hline 311415 & 3 & 5 & 13 & 1597 & 70325856 & 3\\
    \hline 315055 & 131 & 37 & 5 & 13 & 79864160 & 3\\
    \hline 420135 & 3 & 5 & 37 & 757 & 113889888 & 3\\
    \hline 430495 & 179 & 37 & 5 & 13 & 130245856 & 3\\
    \hline 447135 & 3 & 5 & 13 & 2293 & 117673248 & 3\\
    \hline 473415 & 3 & 5 & 37 & 853 & 128212896 & 3\\
    \hline 475215 & 3 & 5 & 13 & 2437 & 131335968 & 3\\
    \hline 490295 & 19 & 13 & 5 & 397 & 129126880 & 3\\
    \hline 504295 & 11 & 173 & 5 & 53 & 167198304 & 3\\
    \hline 545415 & 3 & 5 & 13 & 2797 & 160048800 & 3\\
    \hline 550615 & 43 & 5 & 13 & 197 & 188399904 & 3\\
    \hline 552695 & 11 & 13 & 5 & 773 & 147595872 & 3\\
    \hline 553335 & 3 & 5 & 37 & 997 & 164792736 & 3\\
    \hline 563695 & 11 & 277 & 5 & 37 & 187614944 & 3\\
    \hline 568815 & 3 & 5 & 13 & 2917 & 177955488 & 3\\
    \hline 592215 & 3 & 5 & 13 & 3037 & 186476448 & 3\\
    \hline 603655 & 251 & 37 & 5 & 13 & 209029792 & 3\\
    \hline 606615 & 3 & 5 & 37 & 1093 & 199213344 & 3\\
    \hline 633399 & 3 & 149 & 13 & 109 & 216478368 & 3\\
    \hline 657735 & 3 & 5 & 13 & 3373 & 220949088 & 3\\
    \hline 665223 & 3 & 461 & 13 & 37 & 209098272 & 3\\
    \hline 673215 & 3 & 5 & 37 & 1213 & 233606880 & 3\\
    \hline 685815 & 3 & 5 & 13 & 3517 & 237620448 & 3\\
    \hline 687895 & 19 & 13 & 5 & 557 & 251644512 & 3\\
    \hline 727935 & 3 & 5 & 13 & 3733 & 245403744 & 3\\
    \hline 751335 & 3 & 5 & 13 & 3853 & 264797280 & 3\\
    \hline 755495 & 59 & 13 & 5 & 197 & 237803552 & 3\\
    \hline 757055 & 19 & 13 & 5 & 613 & 232180384 & 3\\
    \hline 790495 & 19 & 53 & 5 & 157 & 312845408 & 3\\
    \hline 798135 & 3 & 5 & 13 & 4093 & 292536288 & 3\\
    \hline 803751 & 3 & 557 & 13 & 37 & 315036576 & 3\\
    \hline 807455 & 11 & 277 & 5 & 53 & 258772704 & 3\\
    \hline 818935 & 43 & 5 & 13 & 293 & 328030688 & 3\\
    \hline 833199 & 3 & 29 & 61 & 157 & 319335264 & 3\\
    \hline 849615 & 3 & 5 & 13 & 4357 & 321461856 & 3\\
    \hline 878415 & 3 & 5 & 157 & 373 & 324090528 & 3\\
    \hline 884455 & 11 & 13 & 5 & 1237 & 376589472 & 3\\
    \hline 886015 & 43 & 5 & 13 & 317 & 395428320 & 3\\
    \hline 886335 & 3 & 5 & 37 & 1597 & 339119712 & 3\\
    \hline 896415 & 3 & 5 & 13 & 4597 & 353689056 & 3\\
    \hline 905255 & 19 & 13 & 5 & 733 & 316252704 & 3\\
    \hline 911495 & 379 & 13 & 5 & 37 & 319484960 & 3\\
    \hline 934935 & 3 & 5 & 157 & 397 & 377527776 & 3\\
    \hline 961935 & 3 & 5 & 13 & 4933 & 386463456 & 3\\
    \hline 973655 & 19 & 37 & 5 & 277 & 357620256 & 3\\
    \hline 981695 & 11 & 13 & 5 & 1373 & 364695008 & 3\\
    \hline 990015 & 3 & 5 & 13 & 5077 & 394553376 & 3\\
    \hline
\end{tabular}
\end{table}
\end{scriptsize}
\end{center}

\bigskip

\noindent Li-Tong Deng,\\
Department of Mathematical Sciences,\\
Tsinghua University, 100084 \\
Beijing, China.\\
{\it denglitong19@gmail.com}

\bigskip

\noindent Yong-Xiong Li,\\
Yau Mathematical Sciences Center,\\
Tsinghua University, 100084 \\
Beijing, China.\\
{\it liyx\_1029@tsinghua.edu.cn}

\end{document}